\documentclass[a4paper,reqno]{amsart}
\usepackage{amssymb}
\usepackage{amscd}
\usepackage{mathrsfs}
\usepackage{color}
\usepackage[all]{xy}

\def\today{\number\day\space\ifcase\month\or   January\or February\or
   March\or April\or May\or June\or   July\or August\or September\or
   October\or November\or December\fi\   \number\year}

\theoremstyle{definition}
\newtheorem{thm}{Theorem}[section]
\newtheorem{lem}[thm]{Lemma}
\newtheorem{prp}[thm]{Proposition}
\newtheorem{dfn}[thm]{Definition}
\newtheorem{cor}[thm]{Corollary}

\newtheorem{rmk}[thm]{Remark}

\newtheorem{exa}[thm]{Example}

\newtheorem{qst}[thm]{Question}
\newtheorem{notation}[thm]{Notation}

\newcommand{\Z}{{\mathbb{Z}}}
\newcommand{\R}{{\mathbb{R}}}
\newcommand{\C}{{\mathbb{C}}}
\newcommand{\N}{{\mathbb{N}}}

\newcommand{\D}{{\mathbb{D}}}
\newcommand{\T}{{\mathbb{T}}}

\newcommand{\ltsr}{{\mathrm{ltsr}}}
\newcommand{\rtsr}{{\mathrm{rtsr}}}
\newcommand{\tsr}{{\mathrm{tsr}}}

\title[Stable rank for inclusions of Banach algebras]{Stable rank for inclusions of Banach algebras\\ \vspace{1cm}
\fontseries{m}\selectfont{\tiny Dedicated to celebrate the 90th birthday of \\ Professor Jun Tomiyama}}

\author{Masaru Nagisa}
\address{Department of Mathematics and Informatics \\Chiba University \\
Yayoi 1-33, Chiba, 263-8522, Japan}

\email[]{nagisa@math.s.chiba-u.ac.jp}

\author{Hiroyuki Osaka}
\address{ Department of Mathematical Sciences\\
  Ritsumeikan University\\ Kusatsu, Shiga, 525-8577  Japan}
\email{osaka@se.ritsumei.ac.jp}
\thanks{$^*$Research of the second author partially supported by the JSPS grant for Scientific Research No.20K03644}

\author{Raisei Tomita}
\address{ Atsugi commercial High School\\
Atsugi,Kanagawa, 243-0817 Japan}
\email[]{rtomita1129@gmail.com}
\email{}

\date{}


\begin{document}

\begin{abstract}
We give a formula for the stable rank of inclusions of unital Banach algebras in the sense of finite Watatani index. As an application we show that  the stable rank of $\ell^1$-algebras of Disk algebras by any action of finite groups is 2.
\end{abstract}

\maketitle

\section{Introduction}

The notion of left topological stable rank (resp. right topological stable rank) for a Banach algebra $A$, denoted by $\mathrm{ltsr}(A)$ (resp. $\mathrm{rtsr}(A)$), was introduced by Rieffel which generalizes the concept of dimension of a topological space \cite{Rieffel 1982}. He presented the basic properties and stability theorem related to K-theory for C*-algebras. Several questions in \cite{Rieffel 1982} has been solved through these 30 years. For example there is a Banach algebra $A$ such that $\mathrm{ltsr}(A) \not= \mathrm{rtsr}(A)$ in \cite{DMLR 2008}.

The Disk algebra $A(\mathbb{D})$ is a typical example in commutative Banach algebras which does not have the involution. It is easy to show that $\mathrm{ltsr}(A(\mathbb{D})) = 2$ and it is deduced that any automorphism $\alpha$ on $A(\mathbb{D})$ can be induced from the holomorphic function $\varphi$ on the Disk $\mathbb{D}$ defined by $\alpha(f)(z) (= \alpha_\varphi(f)(z)) = f(\varphi(z))$ for any $z \in \mathbb{D}$. Therefore, it is interesting to study the dynamical system $(G, A(\mathbb{D}), \alpha_\varphi)$ for a discrete group $G$.

In this note we investigate the inclusion $1 \in A \subset B$ of unital Banach algebras of index-finite type in the sense of Watatani \cite{Wata 1990},  and give several formulas between  $\mathrm{ltsr}(B)$ (resp. $\mathrm{rtsr}(B)$) and $\mathrm{ltsr}(A)$ (resp. $\mathrm{rtsr}(A)$). In particular, $\ell^1$-algebra $\ell^1(G, A, \alpha)$ of a unital Banach algebra $A$ by any action $\alpha$ from a finite group $G$ satisfies 
$\mathrm{ltsr}(\ell^1(G, A, \alpha)) \leq \mathrm{ltsr}(A) + |G| -1$. Moreover, we show that $\mathrm{ltsr}(\ell^1(G, A(\mathbb{D}), \alpha)) = 2$ for any action $\alpha$ from a finite group $G$.

\section{Main result}

Let $1 \in A \subset B$ be an inclusion of unital Banach algebras, and there is an $A$-bimodule map $E\colon B \rightarrow A$ of index-finite type in the sense of Watatani \cite{Wata 1990}. In this section we give a formula for topological stable rank of an inclusion of unital Banach algebras which is the generalized one in the case of an inclusion of unital C*-algebras in \cite{JOPT 2009}.

\vskip 2mm 

\begin{notation}
For any ring $A$ with identity element we let $Lg_n(A)$ $($resp. $Rg_n(A))$ denote the set of $n$-tuples of elements of $A$ which generate $A$ as a left (resp. right) ideal,
\end{notation}

\vskip 2mm

\begin{dfn}
Let $A$ be a unital Banach algebra. By the left (resp. right) topological stable rank of $A$, denote $\mathrm{ltsr}(A)$ (resp. $\mathrm{rtsr}(A)$), we mean the least integer $n$ such that $\mathrm{ltsr}(A)$ (resp. $\mathrm{rtsr}(A)$) is dense in $A^n$. If no such integer $n$ exists we set $\mathrm{ltsr}(A) = \infty$ (resp. $\mathrm{rtsr}(A) = \infty$). If $A$ does not have the identity element, then its topological stable ranks are defined to be those for the Banach algebra ${\tilde A}$ obtained from $A$ by adjoining an identity element.
\end{dfn}

\vskip 2mm

\begin{rmk}
When $A$ has a Banach algebra with a continuous involution, $\mathrm{ltsr}(A) = \mathrm{rtsr}(A)$. However, in general $\mathrm{lstr}(A) \not= \mathrm{rtsr}(A)$. For example, when $\mathcal{N}$ is an atomic nest which is order isomorphic to $\N$ with finite dimensional atoms $E_k$ of rank $n_k$ satisfying $n_k \geq 4\sum_{i<k}n_i$. Then $\mathrm{ltsr}(\mathcal{T}(\mathcal{N})) = \infty$ and $\mathrm{rtsr}(\mathcal{T}(\mathcal{N})) = 2$ \cite[Theorem~2.1]{DMLR 2008}.
\end{rmk}



\vskip 2mm

\begin{rmk}
Let $M_n(\C)$ be an $n\times n$ matrix algebra, $\{E_{ij}\}$ a canonical matrix unit for $M_n(\C)$, and $(A, \|\cdot\|)$ a unital Banach algebra.
Then $(A\otimes M_n(\C), \|\cdot \|_1)$ becomes a Banach algebra, where the norm $\| \cdot \|_1$ is defined as follows:
\[       \| \sum_{i,j=1}^n a_{ij}\otimes E_{ij} \|_1 = \sum_{i,j=1}^n \|a_{ij}\| , \qquad a_{ij}\in A \; (1\le i,j \le n).  \]
\end{rmk}

\vskip 2mm





\vskip2mm

The following Lemma implies that when we consider Banach algebra tensor products of Banach algebras, we have only to consider the $\ell^1$-norm $\|\quad\|_1$.

\vskip 2mm

\begin{lem}\label{lem:subcross norm}
Let $\alpha$ be a norm on $A\otimes M_n(\C)$ and  $(A\otimes M_n(\C), \alpha)$  a unital Banach algebra.
Then the following are equivalent:
\begin{enumerate}
\item[$(1)$] There exists a positive number $C$ such that
\[   \alpha(a\otimes E_{11}) \le C \|a\| \qquad \text{for any } a\in A., \]
\item[$(2)$] The Banach algebra $(A\otimes M_n(\C), \|\cdot\|_1)$ is isomorphic to $(A\otimes M_n(C), \alpha)$ as a Banach algebra.
\item[$(3)$] There exist positive numbers $K$ and $L$ such that
\[  K \max_{1\le i,j \le n} \|a_{ij}\| \le \alpha\left(\sum_{i,j=1}^n a_{ij}\otimes E_{ij}\right) \le L \sum_{i,j=1}^n \|a_{ij}\| \text{ for any } a_{ij}\in A.  \] 
\end{enumerate}
\end{lem}

\begin{proof}
$(1)\Rightarrow (2)$ Set $C= L \max\{ \alpha(1\otimes E_{i1}) \alpha(1\otimes E_{1j}) ; 1\le i,j \le n \}$.
For any $a\in A$, we have
\[   \alpha(a\otimes E_{ij}) = \alpha( (1\otimes E_{i1})(a\otimes E_{11})(1\otimes E_{1j}) \le C \|a\| . \]
Since $\alpha\left(\sum_{i,j=1}^n a_{ij}\otimes E_{ij}\right) \le C \sum_{ij=1}^n \|a_{ij}\| = C\| \sum_{i,j=1}^n a_{ij}\otimes E_{ij} \|_1$,
the mapping ${\rm id}: (A\otimes M_n(\C), \|\cdot\|_1) \longrightarrow (A\otimes M_n(\C), \alpha)$ is continuous.
Since $ (A\otimes M_n(\C), \alpha)$ is complete, this inverse is continuous by the open mapping theorem.
So $(A\otimes M_n(\C), \|\cdot\|_1)$ is isomorphic to $ (A\otimes M_n(\C), \alpha)$.

$(2)\Rightarrow (3)$
Since the norm $\alpha$ is equivalent to the norm $\|\cdot\|_1$, there exist positive numbers $K$ and $L$ such that
\[  K \|\sum_{i,j=1}^n a_{ij}\otimes E_{ij}\|_1 \le \alpha\left(\sum_{i,j=1}^n a_{ij}\otimes E_{ij}\right) \le L \| \sum_{i,j=1}^n a_{ij}\otimes E_{ij}\|_1.  \] 
Using the relation
\[  \max_{1\le i,j \le n}\|a_{ij}\| \le \sum_{i,j=1}^n \|a_{ij}\| \le n^2 \max_{1\le i,j \le n}\|a_{ij}\|,  \]
we have 
\[  K \max_{1\le i,j \le n} \|a_{ij}\| \le \alpha\left(\sum_{i,j=1}^n a_{ij}\otimes E_{ij}\right) \le L \sum_{i,j=1}^n \|a_{ij}\|.  \] 

$(3)\Rightarrow (1)$
It is clear. 
\end{proof}



\vskip 2mm

\begin{dfn}
Let $A$ be a unital Banach algebra. We say that $A$ satisfies Condition $L_n(k)$ if the left-invertible $(n + k) \times k$ matrices are dense in $M_{(n+k)\times k}(A)$.
\end{dfn}

\vskip 2mm



\begin{lem}\label{lem:L_n}
Let $A$ be a unital Banach algebra. If $A$ satisfies $L_n(1)$, then it satisfies $L_n(k)$ for all $k \geq 1$.
\end{lem}

\begin{proof}
The proof is based on \cite[6.3 Lemma]{Rieffel 1982}. We argue by induction on $k$. The assumption holds for the case in which $k=1$, so we suppose that $A$ satisfies condition $L_n(k-1)$, where $k \geq 2$. Let $T=(t_{ij})$ be any element of $M_{(n+k)\times k}(A)$, and let $\varepsilon >0$. According to $L_n(1)$ there exists $(c_i) \in Lg_{n+1}(A)$ such that $\|^t (t_{k1}, t_{(k+1)1}, \dots , t_{(n+k)1}) - (c_i)\| < \varepsilon/2$.

 Set 
\begin{align*}
T' = \begin{pmatrix}
t_{11} & t_{12} & \ldots & t_{1k} \\
\vdots & \vdots &          & \vdots \\
t_{(k-1)1} & t_{(k-1)2} & \ldots & t_{(k-1)k} \\
c_1 & t_{k2} & \ldots & t_{kk} \\
\vdots & \vdots        &     & \vdots \\
c_{(n+1)} & t_{(n+k)2} & \ldots & t_{(n+k)k}
\end{pmatrix} 
\in M_{(n+k)\times k}(A).
\end{align*}


We define a matrix $E_{ij}(a)=(f_{kl})\in M_{n+k}(A)$ for $a\in A$ and  $1\le i \neq j \le N+k$ as follows:
\[   f_{ij} = a, \; f_{kl}=\begin{cases} 1  &  \text{ if } k= l \\
                                              0  & \text{ if } k\neq l \text{ and } (k,l)\neq (i,j) \end{cases}.  \]
Then $E_{ij}(a)^{-1}=E_{ij}(-a)$.
We can choose $d_1,d_2,\ldots d_{n+1} \in A$ such that $\sum_{i=1}^{n+1} d_ic_i =1-t_{11}$ and set
\[  R = (\prod_{l=k}^{n+k} E_{1l} (c_{l-k+1}) ) (\prod_{l=2}^{k-2} E_{1l} (-t_{l1}) ) (\prod_{j=1}^{n+1} E_{1,k-1+j} (d_j) ) . \] 
Then $S=RT'$ has the following form:
\[      S= \begin{pmatrix}  1 & s_{12}  & \ldots & s_{1k} \\
                                  0 &    &          &   \\
                           \vdots &    &  Q      &   \\
                                  0 &    &          &   \end{pmatrix} \in M_{(n+k)\times k}(A) ,  \]              
where $Q\in M_{(n+k-1)\times (k-1)}(A)$.
By the induction hypothesis there is a $Q' \in M_{(n+k-1)\times (k-1)}(A)$ such that $Q'$ is left-invertible and $\|Q-Q'\| < \frac{\varepsilon}{2\|R^{-1}\|}$.




 Set
\begin{align*}
S' = \begin{pmatrix}
1 & s_{12} & \ldots & s_{1k} \\
0 &         &        &         \\
\vdots &    & \Huge{Q'} &  \\
0 &         &        &
\end{pmatrix}
\in M_{(n+k)\times k}(A).
\end{align*}
Since $S'$ is left-invertible, so is $R^{-1}S' \in M_{(n+k)\times k}(A)$. Since $\|T-T'\| < \varepsilon/2$, $\|S-S'\| < \frac{\varepsilon}{2\|R^{-1}\|}$ and $T'=R^{-1}S$, we have $\|R^{-1}S'-T\| \leq \|R^{-1}(S'-S)\|+\|T'-T\| < \varepsilon/2 + \varepsilon/2 = \varepsilon$.
\end{proof}

\vskip 2mm


\begin{dfn}
Let $A$ be a unital Banach algebra and let $p$ and $q$ be idempotents in $A$. 
\begin{enumerate}
\item
We define $p \leq q$ if there is an idempotent $r \in A$ such that $p + r = q$ and $pr = 0 = rp$.
It is obvious that this is equivalent to that $pq = p = qp$.
\item
We call $p$ is equivalent to $q$ if there are elements $x \in pAq$ and $y \in qAp$ such that $xy = p$ and $yx = q$. In this case we write $p \sim q$.
\item
We call $p$ is  subequivalent to $q$ if there is an idempotent $q' \in A$ such that $q' \leq q$ and $p \sim q'$. In this case we write $p \preceq q$. 
\end{enumerate}
\end{dfn}


\vskip 2mm

\begin{rmk}
When $A$  is a unital C*-algebra, then equivalence of two projections $p, q$ in $A$ is the same as the Murray-von Neumann equivalence. Therefore, we can modify \cite[Proposition 3.2, Proposition 4.1, and Proposition 4.2]{blackadar 2004} in the framework of idempotents.
\end{rmk}

\vskip 2mm

Recall that an idempotent $p$ in $A$ is said to be full if it is not contained in any proper closed, two-sided ideal of $A$.

\vskip 2mm

\begin{prp}
Let $A$ be a unital Banach algebra and $p$ a full idempotent in $A$. Then $\ltsr(A) \leq \ltsr(pAp)$.
\end{prp}

\begin{proof}
It follows from the same argument in \cite[Theorem~4.5]{blackadar 2004}.
\end{proof}

\vskip 2mm

\begin{thm}\label{thm:basic}
Let $1 \in A \subset B$ be an inclusion of unital Banach algebras. Suppose that there is an $A$-bimodule  map $E\colon B \rightarrow A$ of index-finite type in the sense of Watatani \cite{Wata 1990}. That is, there exists a quasi-basis $\{(u_i, v_i)\}_{i=1}^n$ in $B \times B$ for $E$ 
for some $n \in \N$ such that for all $b \in B$, $\sum_{i=1}^nu_iE(v_ib) = b = \sum_{i=1}^nE(bu_i)v_i$. 
Then, $\ltsr(B) \leq \ltsr(A) + n-1$ and $\rtsr(B) \leq \rtsr(A) + n-1$. 
\end{thm}

\begin{proof}
Set $m= \ltsr(A) - 1$. We prove that $Lg_{m+n}(B)$ is dense in $B^{m+n}$. Let $b_1,\dots ,b_{m+n} \in B$, and $\varepsilon > 0$. Write
\begin{align*}
b_j = \sum_{k=1}^n E(b_ju_k)v_k 
\end{align*}
for $1 \leq j \leq m+n$.

 Set
\begin{align*} 
a = \begin{pmatrix}
E(b_1u_1) & E(b_1u_2) & \ldots & E(b_1u_n) \\
E(b_2u_1) & E(b_2u_2) & \ldots & E(b_2u_n) \\
\vdots & \vdots &        & \vdots \\
E(b_{m+n}u_1) & E(b_{m+n}u_2) & \ldots & E(b_{m+n}u_n)
\end{pmatrix}
\quad \text{and} \quad
v = \begin{pmatrix}
v_1 \\
v_2 \\
\vdots \\
v_n
\end{pmatrix}.
\end{align*}
This gives
\begin{align*}
av = \begin{pmatrix}
b_1 \\
b_2 \\
\vdots \\
b_{m+n}
\end{pmatrix}.
\end{align*}

According to Lemma \ref{lem:L_n}, the Banach algebra A satisfies the property $L_m(n)$, so that there exists 
\begin{align*}
x = \begin{pmatrix}
x_{11} & x_{12} & \ldots & x_{1n} \\
x_{21} & x_{22} & \ldots & x_{2n} \\
\vdots  & \vdots  &        & \vdots \\
x_{m+n1} & x_{m+n2} & \ldots & x_{m+nn} \\
\end{pmatrix}
\end{align*}
with $x_{jk} \in A$ for $1 \leq j \leq m+n$ and $1 \leq k \leq n$, such that $\|x-a\| < \varepsilon / \|v\|$ and $x$ is left-invertible. That is, there is a $n \times (m+n)$ matrix $z$ with entries in A, such that $zx = 1_n$.

Define
\begin{align*}
y_j = \sum_{k=1}^n x_{j,k}v_k \in B
\end{align*}
for $1 \leq j \leq m+n$. Then
\begin{align*}
xv = \begin{pmatrix}
y_1 \\
y_2 \\
\vdots \\
y_{m+n}
\end{pmatrix}.
\end{align*}
Setting $y = xv$, we see that $v = (zx)v = z(xv) = zy$. Hence 
$(E(u_1) \dots E(u_n))zy = (E(u_1) \dots E(u_n))v = \sum_{k=1}^n E(u_k)v_k = 1$. Therefore $y \in Lg_{m+n}(B)$. We have $\|xv-av\| \leq \|x-a\|\|v\| < \varepsilon$, and this proves that $Lg_{m+n}(B)$ is dense in $B^{m+n}$.
\end{proof}

\vskip 2mm

\begin{exa}\label{exa:ell}
Let $G$ be a discrete group, $A$ be a unital Banach algebra, and $\alpha:A\longrightarrow A$ be a bounded  
algebra isomorphism of $A$ with bounded inverse $\alpha^{-1}$.
We denote $\alpha \in {\rm Aut}(A)$.
We call an action $\alpha : G\ni g \mapsto \alpha^g \in {\rm Aut}(A)$ of $G$ on $A$ if  $\alpha^g \in {\rm Aut}(A)$ satisfies
\[   \sup_{g\in G} \|\alpha^g\|<\infty, \; \alpha^e={\rm id}, \text{ and } \alpha^{gh}=\alpha^g \alpha^h \ \hbox{for any}\ g, h \in G,\]
where $e$ is a unit of $G$.

We denote $f\in \ell^1(G,A, \alpha)$ if $f:G\longrightarrow A$ and
\[   \sum_{g\in G} \|f(g)\| <\infty.  \]
For any $f_1,f_2\in \ell^1(G,A,\alpha)$ we define the product $f_1*f_2$ and the norm $\|f_1\|$ as follows:
\[  f_1*f_2(g) = \sum_{h\in G} f_1(h)\alpha^h(f_2(h^{-1}g)), \; \|f_1\|=\sum_{g\in G} \|f_1(g)\|.  \]
Then $\ell^1(G,A,\alpha)$ becomes unital Banach algebra.
The element $f\in \ell^1(G, A,\alpha)$ can be represented as follows:
\[  f=\sum_{g\in G} f_g \delta^g, \text{ and } \|f\| = \sum_{g\in G} \|f_g\| , \]
where $f_g=f(g)\in A$,  $\delta_g \in \ell^1(G,A,\alpha)$ satisfies
\[   (f_g \delta^g)*(\tilde{f}_h\delta^h) = f_g \alpha^g(\tilde{f}_h)\delta^{gh} \quad (f,\tilde{f}\in \ell^1(G,A,\alpha)) ,   \]
and $\delta^e$ is the unit of $\ell^1(G,A,\alpha)$.
We can see $A (\cong A\delta^e)$ as a unital Banach subalgebra of $\ell^1(G,A,\alpha)$.

\vspace{5mm}

%

We define a bounded linear map $E:\ell^1(G,A, \alpha) \longrightarrow \ell^1(G,A,\alpha)$ as follows:
\[   E(f) = E\left(\sum_{g\in G}f_g \delta^g\right) = f_e\delta^e  .\]
Then, for any $a\in A(\cong A\delta^e)$, we have
\[   E^2(f)=E(f),  \; E(af)=aE(f), \; \text{ and } E(fa)=E(f)a  . \]
Set $u_g=\delta^g$ and $v_g=\delta^{g^{-1}}$ for $g\in G$.

\begin{lem}\label{lem:quasi-basis}
For any $b= \sum_{g\in G} b_g \delta^g$ with $| \{g\in G: b_g\neq 0\} | <\infty$, we have
\[  b = \sum_{g\in G}u_g *E(v_g * b) = \sum_{g\in G}E(b * u_g) * v_g. \]
In particular, if $|G|$ is finite, we have $\sum_{g\in G} u_g * v_g =|G| \delta^e$. Here, $|A|$ denotes the cardinal number of a set $A$.
\end{lem}

\begin{proof}
\begin{align*}
   \sum_{g\in G} u_g * E(v_g*b) & = \sum_{g\in G}\delta^g* E\left(\delta^{g^{-1}}* \sum_{h\in G} b_h \delta^h\right)
                = \sum_{g\in G}\delta^g * E\left(\sum_{h\in G} \alpha^{g^{-1}}(b_h) \delta^{g^{-1}h}\right)  \\        
        & =  \sum_{g\in G} \delta^{g} * \alpha^{g^{-1}}(b_g) \delta^{e} =\sum_{g\in G} b_g \delta^g =b .
\end{align*}
\end{proof}

If $G$ is a finite group, then we know that an inclusion $A \subset \ell^1(G, A, \alpha)$ is of index-finite type in the sense of Watatani from Lemma~\ref{lem:quasi-basis}. Indeed, $\{(u_g, v_g)\}_{g\in G}$ is  a quasi-basis for $E$ and index of $E$ is $\mathrm{Index}E = \sum_{g\in G}{\delta^g}^{-1}\delta^g = |G|\delta^e$. Note that $E$ is contractive.
From Theorem~\ref{thm:basic} 
$$
\ltsr(\ell^1(G, A, \alpha)) \leq \ltsr(A) + |G|-1.
$$ 

Moreover, if $G$ is a finite group ($|G|=n$),
we define a linear operator $e_{g,h}$ on $\C^N\cong \oplus_{g\in G} \C g$ as follows:
\[   e_{g,h} k = \begin{cases} g \quad & (h=k) \\ 0 & (h\neq k) \end{cases},  \]
where  $g,h,k\in G$.
Then $\{e_{g,h} : g,h\in G\}$ generates the matrix algebra $M_n(\C)$.
We can define an algebra-isomorphism as follows:
\[  \pi:\ell^1(G,A,\alpha) \ni f = \sum_{g\in G} f_g\delta^g \mapsto 
         \pi(f) = \sum_{h,g\in G} \alpha^h(f_g)\otimes e_{h,hg}\in A\otimes M_n(\C).  \]
Using the correspondence, we can see $\ell^1(G, A, \alpha)$ as a closed subalgebra of $A\otimes M_n(\C)$.
\end{exa}

\vskip 2mm

\begin{lem}\label{lem:tsr matrix}
Let $A$ be a unital Banach algebra and $m\in \N$.
For the Banach algebra $(M_{m}(A)(=A\otimes M_{m}(\C)), \|\cdot\|_1)$, 
\[   \ltsr(M_m(A)) = \left\{\frac{\ltsr(A) -1}{m}\right\} + 1,  \]
where $\{\quad \}$ denotes `the least integer greater than'. The same estimate holds for $\rtsr$. 

\end{lem}

\vskip 2mm

\begin{proof}
It follows the same step in \cite[6.1 Theorem]{Rieffel 1982}.
\end{proof}

\vskip 2mm

\begin{cor}\label{cor:matrix}
Let $B$ be a unital Banach algebra. Suppose that there exists a matrix units $\{e_{ij}\}_{i,j=1}^n$ in $B$. Then
\begin{enumerate}
\item There exists an algebra-isomorphism $\rho\colon B \rightarrow (e_{11}Be_{11})\otimes M_n(\C)$,
\item
$\tsr(B) = \left\{\frac{\tsr(e_{11}Be_{11}) -1}{n}\right\}+1$. 
Here a matrix units $\{e_{ij}\}$ in $B$ means that $\sum_{i=1}^ne_{ii} = 1$ and $e_{ij}e_{kl} = \delta_{jk}e_{il}$ for $1 \leq i, j, k, l \leq n$.
\end{enumerate}
\end{cor}

\begin{proof}

(1) 
Define $\rho\colon B \rightarrow (e_{11}Be_{11})\otimes M_n(\C)$ by
$\rho(b) = \sum_{i,j=1}^n e_{1i}be_{j1} \otimes E_{ij}$. 
Then, it is obvious that $\rho$ is an isomorphism from the direct calculus. 

(2)
Define a norm $\alpha$ on $(e_{11}Be_{11})\otimes M_n(\C)$ by 
\[  \alpha\left(\sum_{i,j=1}^n b_{ij}\otimes E_{ij}\right)= \| \rho^{-1} (\sum_{i,j=1}^n b_{ij} \otimes E_{ij}) \|
               = \| \sum_{i,j=1}^n e_{i1} b_{ij} e_{1j} \|,    \]
where $b_{ij} \in e_{11}Be_{11}$ $(i,j=1,2,\ldots, n)$.       
Then $\alpha(b\otimes E_{11}) = \|e_{11}be_{11} \| = \|b\|$ for $b\in e_{11}Be_{11}$.
By Lemma~\ref{lem:subcross norm} $(e_{11}Be_{11}\otimes M_n(\C), \alpha)$ is isomorphic to $(e_{11}Be_{11}\otimes M_n(\C), \|\cdot\|_1)$ as a Banach algebra.
So we have that 
$\tsr(B) = \left\{\frac{\tsr(e_{11}Be_{11}) -1}{n}\right\}+1$ by Lemma \ref{lem:tsr matrix}.
\end{proof}

\vskip2mm

\begin{cor}
Under the same assumption in Theorem~\ref{thm:basic} suppose that 
\begin{enumerate}
\item[$1)$] an $\mathrm{Index}E = \sum_{i=1}^nu_iv_i$ is invertible, 
\item[$2)$] $E$ is contractive, 
\item[$3)$] a quasi-basis $\{(u_i, v_i)\}_{i=1}^n$ is orthogonal with respect to $E $, i.e., $E(u_iv_i) = \delta_{i,j}$ for $1 \leq i, j \leq n$.
\end{enumerate}
Then 

$$
\ltsr(A) \leq n(\ltsr(B)) + n^2 - n + 1.
$$

Moreover, $\ltsr(A)$ is finite if and only if $\ltsr(B)$ is finite. The same is true for rtsr.
\end{cor}

\begin{proof}
Let $\langle B, e_A\rangle = \{\sum_lx_le_Ay_l\colon x_l , y_l \in B\}$ be the basic construction derived from the inclusion and $E_2\colon \langle B, e_A\rangle \rightarrow B$ be the dual conditional expectation. Since $\{(u_ie_A(\mathrm{Index}E), e_Av_i)\}_{i=1}^n$ is the quasi-basis for $E_2$ by \cite[Proposition 1.6.6]{Wata 1990}. Hence, by Theorem \ref{thm:basic} 
$$
\ltsr(\langle B, e_A\rangle) \leq \ltsr(B) + n -1.
$$

On the contrary, since a quasi-basis $\{(u_i, v_i)\}_{i=1}^n$ is orthogonal with respect to $E$, by \cite[Lemma 3.3.4]{Wata 1990} there exists an isomorphism $\phi\colon \langle B, e_A\rangle \rightarrow M_n(A)$ by $\phi(xe_Ay) = [E(v_ix)E(yu_j)]$ for $x, y \in B$. 

Then we define a norm $\|\cdot \|_{\phi}$ on $A\otimes M_n(\C)$ by $\| \phi(X) \|_{\phi} = \|X\|_{op}$ ($X\in \langle B, e_A\rangle$),  
where $\|\cdot\|_{op}$ denotes an operator norm on $B$ (i.e., the norm on $\langle B, e_A\rangle$).
So $(A\otimes M_n(\C), \|\cdot\|_{\phi})$ is a Banach space.

Since $\phi(u_1ae_Av_1)=\sum_{i,j=1}^n  E(v_iu_1a)E(v_1u_j)\otimes E_{ij} = a\otimes E_{11} $  for $a\in A$,
\[   \|a\otimes E_{11}\|_{\phi} =\|(u_1a)e_Av_1\|_{op} \le \|a\| \|u_1\| \|v_1\|.  \]
By Lemma~\ref{lem:subcross norm}  ,  $(A\otimes M_n(\C), \|\cdot\|_{\phi})$ is isomorphic to $(A\otimes M_n(\C), \|\cdot\|_1)$.
Therefore, 
we have $\ltsr(\langle B, e_A\rangle) = \left\{\frac{\ltsr(A) - 1}{n}\right\} + 1$ by Lemma \ref{lem:tsr matrix}. 

Hence,
\begin{align*}
\frac{\ltsr(A) - 1}{n} &\leq \ltsr(\langle B, e_A\rangle) \leq \ltsr(B) + n-1
\quad (\hbox{by Theorem~\ref{thm:basic}}).\\
\end{align*}
\hbox{Therefore}, 
\begin{align*}
\ltsr(A) &\leq n(\ltsr(B)) + n^2 - n + 1.
\end{align*}
%
%
%
\end{proof}

\color{black}

\vskip 2mm



\vskip 2mm

We say that  a Banach algebra $A$ is tsr boundedly divisible (cf.\cite[Definition~4.2]{Rieffel 1987}) if there is a constant $K > 0$ such that for every positive integer $m$ there is an integer $n \geq m$ such that $A$ can be expressed as $M_n(D)$ for a Banach algebra $D$ with $\tsr(D) \leq K$.

\vskip 2mm

\begin{rmk}
Let $A$ be the tsr boundedly divisible. Then $\ltsr(A) \leq 2$. Indeed, from the definition there is a constant $K > 0$ such that for every positive integer $m$ there is an integer $n \geq m$ such that $A$ can be expressed as $M_n(D)$ for a Banach algebra $D$ with $\tsr(D) \leq K$. Take $m \geq K$. Then there is an integer $n \geq m \geq K$ such that $A \cong M_n(D)$. From Lemma~\ref{lem:tsr matrix} 
$$
\ltsr(A) = \left[\frac{\tsr(D) -1}{n}\right] + 1 \leq \left[\frac{K-1}{K}\right] + 1 \leq 2.
$$
\end{rmk}

\vskip 2mm

As in the case of factor-subfactor theorem (\cite{GHJ 1989}), the inclusion $1 \in A \subset B$ of unital Banach algebras of index-finite type is said to have depth $k$ if the derived tower obtained by iterating the basic construction
$$
A' \cap A \subset A' \cap B \subset A' \cap B_2 \subset A' \cap B_3 \subset \cdots
$$
satisfies $(A' \cap B_k)e_k(A' \cap B_k) = A' \cap B_{k+1}$, where $\{e_k\}_{k\geq 1}$ are projections derived obtained by iterating the basic construction such that $B_{k+1} = \langle B_k, e_k\rangle $ $(k \geq 1)$ $(B_1 = B, e_1 = e_A)$.

\vskip 2mm

\begin{exa}\label{exa:depth2}
When $G$ is a finite group and $\alpha$ an action of $G$ on a Banach algebra $A$, it is obvious that an inclusion $1 \in A \subset \ell^1(G, A, \alpha)$ is of depth 2 (see \cite[Lemma 3.1]{OT 2006}).
\end{exa}

\color{black}

\vskip 2mm

\begin{cor}\label{cor:tsr boundedly}
Let $1 \in A \subset B$ be an inclusion of unital Banach algebras. Suppose that $A$ is tsr boundedly divisible with $\ltsr(A) = 1$, there is an $A$-bimodule  map $E\colon B \rightarrow A$ of index finite type in the sense of Watatani. 
Suppose that $1)$ an $\mathrm{Index}E = \sum_{i=1}^nu_iv_i$ is invertible, $2)$$E$ is contractive, and $3)$ the inclusion $1 \in A \subset B$ is of depth $2$.
Then $B$ is tsr boundedly divisible. Moreover, 
$$
\ltsr(B) \leq 2.
$$
\end{cor}

\begin{proof}
It follows from the same step in \cite[Theorem 5.1]{OT 2006} and Corollary~\ref{cor:matrix}.
\end{proof}

\vskip 2mm

\begin{cor}
Let $G$ be a finite group, $A$ be a unital C*-algebra, and $\alpha \colon G \rightarrow \mathrm{Aut}(A)$ be an action of $G$ as  *-automorphisms of $A$. Suppose that $A$ is a C*-algebra and tsr boundedly divisible with $\tsr(A) = 1$. Then,
$$
\tsr(\ell^1(G, A, \alpha)) \leq 2.
$$
\end{cor}

\begin{proof}
From Example~\ref{exa:ell} the canonical conditional expectation $E\colon \ell^1(G, A, \alpha) \rightarrow A$ is contractive and $\mathrm{Index}E = |G|\delta^e$ is invertible. Moreover, from Example~\ref{exa:depth2} $A \subset \ell^1(G, A, \alpha)$ is of depth 2. The conclusion comes from Corollary~\ref{cor:tsr boundedly}. 
\end{proof}

\vskip 2mm


For C*-algebras  $A$ and $B$, $A \otimes B$ means a C*-algebra with the minimal C*-tensor norm. 

\vskip 2mm

\begin{exa}
Let $A$ be a unital C*-algebra and $B$ an uniform hyperfinite algebra. (see \cite{Davidson 1996} for the definition of UHF algebras) Then $A \otimes B$ becomes tsr boundedly divisible using an appropriate norm which satisfies the condition in Lemma~\ref{lem:tsr matrix}. Then, for any action $\alpha$ from a finite group $G$ on $A \otimes B$ we conclude that $\ltsr(\ell^1(G, A \otimes B, \alpha)) \leq 2$.
\end{exa}

\vskip 2mm
\section{$\ell^1$-algebras $\ell^1(G, A, \alpha)$}

Recall that $C^*(G, A, \alpha)$ denotes the C*-crossed product algebra by an action $\alpha$ from a discrete group $G$ on a unital C*-algebra $A$.
We recommend the reference \cite{GKPT 2018} for it.

\vskip 2mm

The following comes from basic facts.

\vskip 2mm

\begin{lem}\label{lem:basic}
Let $G$ be a finite group, $A$ be a unital C*-algebra, and $\alpha \colon G \rightarrow \mathrm{Aut}(A)$ be an action of $G$ as  *-automorphisms of $A$. We have, then,
$$
\tsr(C^*(G, A, \alpha)) \leq \tsr(\ell^1(G, A, \alpha)).
$$
\end{lem}

\begin{proof}
Since the C*-crossed product algebra $C^*(G, A, \alpha)$ is the enveloping C*-algebra of $\ell^1(G, A, \alpha)$, the canonical continuous map from $\ell^1(G, A, \alpha)$ to $C^*(G, A, \alpha)$ has dense range. Hence, we get the conclusion.
\end{proof}

\vskip 2mm

\begin{exa}
There is a symmetry $\alpha$ on CAR algebra $\otimes_{n=1}^\infty M_2(\C)$ such that 
$C^*(\Z/2\Z, \otimes_{n=1}^\infty M_2(\C), \alpha) (= A)$ has the non-trivial $K_1$-group (\cite{blackadar 1990}). Hence, $\tsr(C[0, 1] \otimes A) = 2$. Let $\beta = id \otimes \alpha$ be an action from $\Z/2\Z$ on $C[0, 1] \otimes (\otimes_{n=1}^\infty M_2(\C))$. Then, $C^*(\Z/2\Z, C[0, 1] \otimes (\otimes_{n=1}^\infty M_2(\C)), \beta) \cong C[0,1] \otimes A$. Hence, from Theorem~\ref{thm:basic} and Lemma~\ref{lem:basic}
\begin{align*}
2 & =  \tsr(C^*(\Z/2\Z, C[0, 1] \otimes (\otimes_{n=1}^\infty M_2(\C)), \beta)) \\
  & \leq \tsr(\ell^1(\Z/2\Z, C[0, 1] \otimes (\otimes_{n=1}^\infty M_2(\C)), \beta))\\
&\leq \tsr(C[0, 1] \otimes (\otimes_{n=1}^\infty M_2(\C))) + 1 = 2.
\end{align*}
Hence, $\tsr(\ell^1(\Z/2\Z, C[0, 1] \otimes (\otimes_{n=1}^\infty M_2(\C)), \beta)) = 2$.
\end{exa}

\vskip 2mm

\begin{qst}
Is there a simple, unital, infinite dimensional C*-algebra with $\tsr(A) = 1$ and an action $\alpha$ of a finite group $G$ on $A$ such that $\tsr(\ell^1(G, A, \alpha)) = 1$?
\end{qst}

\vskip 2mm

Recall that a unital C*-algebra $A$ has property (SP) if for any nonzero positive element $a \in A$, $\overline{aAa}$ has a nonzero projection.

\vskip 2mm

\begin{prp}
Let $A$ be a simple, unital, infinite dimensional C*-algebra with $\tsr(A) = 1$ and property (SP)
and $\alpha$ an action of a finite group $G$ on $\mathrm{Aut}(A)$. Then
$\tsr(\ell^1(G, A, \alpha)) \leq 2$.
\end{prp}

\begin{proof}
It follows from the same argument in Corollary 3.3 in \cite{OT 2007}. 
\end{proof}

\vskip 2mm


\vskip 2mm

\begin{cor}\cite[Theorem~3.2]{OT 2007}
Let $A$ be a simple, unital, infinite dimensional C*-algebra with $\tsr(A) = 1$ and property (SP)
and $\alpha$ an action of a finite group $G$ on $\mathrm{Aut}(A)$. Then
$\tsr(C^*(G, A, \alpha)) \leq 2$.
\end{cor}

\vskip 2mm

As in the same argument in \cite{Rieffel 1982} we can estimate the stable rank of $\ell^1(\Z, A, \alpha)$ for a unital Banach algebra $A$. 

\vskip 2mm

\begin{prp}\label{thm:integer}
Let $A$ be a unital Banach algebra and $\alpha\colon \Z \rightarrow \mathrm{Aut}(A)$ be an action from an integer group $\Z$ on $A$. Then, $\ltsr(\ell^1(\Z, A, \alpha)) \leq \ltsr(A) + 1$.
\end{prp}
\begin{proof}
It follows from the same argument in \cite[Theorem~7.1]{Rieffel 1982}.
\end{proof}

\vskip 2mm

The following seems to be well-known.

\vskip 2mm

\begin{prp}\label{prp:extension}
Let $G$ and $K$ be discrete groups and $\rho\colon G \rightarrow K$ be a surjective group homomorphism.  Let $A$ be a unital Banach algebra and $\alpha\colon K \rightarrow \mathrm{Aut}(A)$ be an action $\alpha$ from $K$ on $A$. 
\begin{enumerate}
\item
There exists an action $\tilde{\alpha}$ from $G$ on $A$ such that for $g \in G$,  $\tilde{\alpha}^g (a) = \alpha^{\rho(g)} (a) $ for any $a \in A$,
\item
There exists a surjective homomorphism from $\ell^1(G, A, \tilde{\alpha})$ to $\ell^1(K, A, \alpha)$.
\end{enumerate}
\end{prp}

\begin{proof}
It is trivial.
\end{proof}

\vskip 2mm

\begin{cor}\label{cor:estimate}
Let $A$ be a unital Banach algebra and $\alpha\colon \Z/n\Z \rightarrow \mathrm{Aut}(A)$ be an action from a cyclic group $\Z/n\Z$ on $A$ $( n \geq 2)$. Then, $\ltsr(\ell^1(\Z/n\Z, A, \alpha)) \leq \ltsr(A) + 1$.
\end{cor}

\begin{proof}
Since there is a surjective homomorphism from $\Z$ to $\Z/n\Z$, by Proposition~\ref{prp:extension} there is a surjective homomorphism from $\ell^1(\Z, A, \tilde{\alpha})$ to $\ell^1(\Z/n\Z, A, \alpha)$. Hence, by Proposition~\ref{thm:integer} 
\begin{align*}
\ltsr(\ell^1(\Z/n\Z, A, \alpha)) &\leq \ltsr(\ell^1(\Z, A, \tilde{\alpha}))\\
&\leq \ltsr(A) + 1.
\end{align*}
\end{proof}

\vskip 2mm

\begin{rmk}
Since $\ltsr(A(\mathbb{D})) =2$, we can get $\ltsr(\ell^1(\Z/n\Z, A(\mathbb{D}), \alpha)) \leq 3$ for any action $\alpha$ from any finite cyclic group $\Z/n\Z$ from Corollary~\ref{cor:estimate}. In the next section we shall prove that $\ltsr(\ell^1(\Z/n\Z, A(\mathbb{D}), \alpha)) = 2$.
\end{rmk}

\section{Disk algebra}
In this section we consider a Disk algebra $A(\mathbb{D})$, which is famous in commutative Banach algebras, and try to determine the stable rank of $\ell^1(\Z/n\Z, A(\mathbb{D}), \alpha)$ by an action $\alpha$ from a finite cyclic group $\Z/n\Z$ for any $n \in \N$ ($n \geq 2$).

\vskip 2mm

We denote $f\in A(\D)$, which means 
\[  f(z) =\sum_{k=0}^\infty c_kz^k \quad (c_k\in \C, k=0,1,2,\ldots) ,  \]
$f(z)$ is holomorphic on $\D=\{z\in \C: |z|<1\}$ and continuous on $\bar{\D}$.
The commutative algebra $A(\D)$ is a unital Banach algebra with the norm $\|f\|=\sup_{z\in \bar{\D}} |f(z)|$.
For $n\in \N$, define the action $\alpha$ of $\Z/n\Z=\{0,1,2,\ldots, n-1\}$ on $A(\D)$ as follows:
\[   \alpha(f)(z) = \sum_{k=0}^\infty c_k \omega^k z^k ,  \]
where $\omega \in \C$ satisfies $\omega^n=1$.

\vskip 2mm

Set $B_n = \ell^1(\Z/n\Z, A(\D), \alpha)$.
For $a\in B_n$, we denote $a_i \in A(\D)$ $(i=0,1,\ldots, n-1)$ such that 
\[     a=\sum_{i=0}^{n-1} a_i\delta^i   \]
and construct $a^{(n-1,1)}\in B_n$ for $a$ such that $a^{(n-1,1)} \in A(\D)\delta^0$ by the following procedure.
Set 
\[   a^{(1,j)} = (\alpha^j(a_0) \delta^0 - a_j \delta^j )*a \quad (j=1,2,\ldots,n-1).  \]
Then we have
\[   a^{(1,j)} =\sum_{i=0}^{n-1} a^{(1,j)}_i\delta^i = \sum_{i=0}^{n-1} (\alpha^j(a_0)a_i - a_j\alpha^j(a_{i-j})) \delta^i   \]
and $a^{(1,j)}_j = 0$.


We also define
\begin{align*}
  a^{(2,j)} & = a^{(1,j+1)}_j \delta^j*a^{(1,1)} - (\delta^j *a^{(1,1)})_j \delta^0* a^{(1,j+1)} \quad (j=1,2,\ldots,n-2) \\
  a^{(3,j)} & = a^{(2,j+1)}_j \delta^j*a^{(2,1)} - (\delta^j *a^{(2,1)})_j \delta^0* a^{(2,j+1)} \quad (j=1,2,\ldots,n-3) \\
             &   \cdots  \\
  a^{(n-2,j)} & = a^{(n-3,j+1)}_j \delta^j*a^{(n-3,1)} - (\delta^j *a^{(n-3,1)})_j \delta^0* a^{(n-3,j+1)} \quad (j=1,2) \\
  a^{(n-1,1)} & = a^{(n-2,2)}_1 \delta^1*a^{(n-2,1)} - (\delta *a^{(n-2,1)})_1 \delta^0* a^{(n-2,2)}.
\end{align*}
Then we have $a^{(k,j)} \in B_n * a$ and
\[   a^{(k,j)}_j = a^{(k,j)}_{j+1} = \cdots =a^{(k,j)}_{j+k-1} = 0 \quad  (1\le k\le n-1,  1\le j \le n-k).  \]
By calculation
\begin{align*}
   a^{(2,j)} = & \sum_{i=0}^{n-1} ( (\alpha^{j+1}(a_0)a_j - a_{j+1}\alpha^{j+1}(a_{-1}))
        (\alpha^2(a_0)\alpha(a_{i-1}) -\alpha(a_1)\alpha^2(a_{i-2}))  \\
                &  - (\alpha^2(a_0)\alpha(a_{j-1}) -\alpha(a_1)\alpha^2(a_{j-2}))
        (\alpha^{j+1}(a_0)a_i -a_{j+1}\alpha^{j+1}(a_{i-j-1}) ) ) \delta^i.  
\end{align*}
In the case $n=2$, we have
\[   a^{(1,1)} = (\alpha(a_0)a_0 -a_1\alpha(a_1) )\delta^0   \]
and in the case $n=3$,
\begin{align*}
   a^{(2,1)} = & ( (\alpha^2(a_0)a_1-a_2\alpha^2(a_2))(\alpha^2(a_0)\alpha(a_2)-\alpha(a_1)\alpha^2(a_1))  \\
     &  \qquad -(\alpha^2(a_0)\alpha(a_0)-\alpha(a_1)\alpha^2(a_2))(\alpha^2(a_0)a_0 -a_2\alpha^2(a_1)) )\delta^0.
\end{align*}

\vskip 2mm

\begin{lem}\label{lem:polynomial}
Let $a=\sum_{i=0}^{n-1} a_i\delta^i \in B_n$,  $(a_0, a_1,\ldots, a_{n-1}\in A(\D) )$.
Then $a^{(n-1,1)} =a^{(n-1,1)}_0 \delta^0$ and $a^{(n-1,1)}_0 \in A(\D)$ is a homogeneous polynomial of degree $2^{n-1}$ 
with $n^2$ variables $\alpha^i (a_j)$ $(i,j=0,1,2,\ldots,n-1)$ and contains only a term which has the form 
$\alpha^{j_1}(a_0)\alpha^{j_2}(a_0)\cdots \alpha^{j_{2^{n-1}}}(a_0)$, $(j_k \in \{0,1,2,\ldots, n-1\})$.
\end{lem}

\begin{proof}
By definition, for any $i\in \{1,2,\ldots,n-1\}$ and $j\in \{1,2,\ldots, n-1\}$, $a^{(1,j)}_0$ contains only a term 
$\alpha^{j_1}(a_0)\alpha^{j_2}(a_0)$ for some $j_1,j_2$ and $a^{(1,j)}_i$ does not contain any term of the form 
$\alpha^{j_1}(a_0)\alpha^{j_2}(a_0)$.

We assume that, for any $i\in \{1,2,\ldots,n-1\}$ and $j\in \{1,2,\ldots, n-k\}$, $a^{(k,j)}_0$ contains only a term $\alpha^{j_1}(a_0)\alpha^{j_2}(a_0)\cdots \alpha^{j_{2^k}}(a_0)$ for some $j_1,j_2,\ldots,j_{2^k}$ and $a^{(k,j)}_i$ does not contain any term of the form $\alpha^{j_1}(a_0)\alpha^{j_2}(a_0)\cdots \alpha^{j_{2^k}}(a_0)$.
Then we have, for any $i\in \{1,2,\ldots,n-1\}$ and $j\in \{1,2,\ldots, n-k-1\}$, $a^{(k+1,j)}_0$ contains only a term $\alpha^{j_1}(a_0)\alpha^{j_2}(a_0)\cdots \alpha^{j_{2^{k+1}}}(a_0)$ for some $j_1,j_2,\ldots,j_{2^{k+1}}$ and $a^{(k+1,j)}_i$ does not contain any term of the form 
$\alpha^{j_1}(a_0)\alpha^{j_2}(a_0)\cdots \alpha^{j_{2^{k+1}}}(a_0)$,
using the relation
\[   a^{(k+1,j)}  = a^{(k,j+1)}_j \delta^j*a^{(k,1)} - (\delta^j *a^{(k,1)})_j \delta^0* a^{(k,j+1)} \]
$j=1,2,\ldots,n-k-1$.

So the statement can be  proved by the induction on $k$.
\end{proof}



\begin{lem}\label{lem:approximation}
Let $a=\sum_{i=0}^{n-1} a_i\delta^i \in B_n$.
For any positive number $\epsilon$ and a finite set $S =\{z_1,z_2,\ldots, z_N\}\subset \C$,
we can choose $b=\sum_{n=0}^{n-1} b_i\delta^i \in B_n$ such that
each $b_i$ $(i=0,1,\ldots,n-1)$ is a polynomial of $z$, $\|a - b\|<\epsilon$,  and
\[   \{z\in \C: b^{(n-1,1)}_0(z) = 0 \} \cap S = \phi.   \]
\end{lem}

\begin{proof}
We can choose $c=\sum_{i=0}^{n-1} c_i\delta^i \in B_n$ such that $c_i$ $(i=0,1,\ldots, n-1)$ is
a polynomial on $z$ in $A(\D)$ and $\|a-c\|<\epsilon/2$.
Let $C$ be a constant term of polynomial $c_0$.
By Lemma~\ref{lem:polynomial}, the constant term of the polynomial $c^{(n-1,1)}_0$ has the following form:
\[   C^{2^{n-1}} + (\text{ each term contains } C^i \; (0\le i <2^{n-1}) ).  \]
Since the solutions of $c^{(n-1,1)}_0(z)=0$ move continuously with respect to coefficients of $c_0^{(n-1,1)}$, 
we can choose a polynomial $b_0$ satisfying $\|b_0-c_0\|<\epsilon/2$
and
\[  \{z\in \C: b^{(n-1,1)}_0(z) = 0 \} \cap S = \phi ,  \]
where $b=\sum_{i=0}^{n-1} b_i \delta^i \in B_n$ and $b_i=c_i$ $(i=1,2,\ldots, n-1)$. 
Then $\|a-b\| \le \|a-c\|+\|c-b\|<\epsilon$.
\end{proof}


\vspace{5mm}

\begin{thm}\label{thm:disc algebra}
\[  {\rm ltsr}(\ell^1(\Z/n\Z, A(\D), \alpha)) =2.  \]
\end{thm}

\begin{proof}
Set $B_n = \ell^1(\Z/n\Z, A(\D), \alpha)$.
At first we show that ${\rm ltsr}(B_n)>1$.
By the identification of $B_n$ with $A(\D)\otimes M_n(\C)$ using $\pi$,
$B_n$ can be embedded in $C(\bar{\D}, M_n)$ as follows:
\[   B_n \ni b =\sum_{i=0}^{n-1} b_i\delta^i \mapsto \pi(b)(z) = \sum_{i,j=0}^{n-1}\alpha^i(b_{j-i})(z)e_{i,j} \in M_n(\C) \]
for any $z\in \bar{\D}$. 
We choose an element $a\in A(\D)$ as a function $a(z)=z$ and a positive number $\epsilon$ 
satisfying with $(1-\epsilon)^2>n!\epsilon^2$.
If we assume that  ${\rm ltsr}(B_n)=1$, then there exists a $b=\sum_{i=0}^{n-1} b_i \delta^i \in B_n$ 
such that 
\[  \| b - a\delta^0\| = \|b_0-a\|+ \sum_{i=1}^{n-1} \|b_i\| < \epsilon \]
and $c*b=\delta^0$ for some $c\in B_n$.

When $|z|=1$, we have
$|a(z)|=1$, $ |b_0(z)|>1-\epsilon$, $|b_i(z)|<\epsilon$   $(i=1,2,\ldots, n-1)$.
So $\det( \pi(b)(z) ) \neq 0$, for any $|z|=1$.
Since $\pi(c*b)(z) = (\pi(c)\pi(b) )(z)=\pi(c)(z) \pi(b)(z)$, 
we also have
\[    \det(\pi(c)(z)) \det(\pi(b)(z)) = \det(\pi(\delta^0)(z)) = 1 \quad (z\in \D) .  \]
That is,  $\det(\pi(b)(z)) \neq 0$ for any $z\in \bar{\D}$.
We define a family of continuous functions on $\T=\{z\in\C: |z|=1\}$ to $\T$ as follows:
\begin{align*}
   f_t(z) & = \frac{\det( \pi( (1-t)a + tb)(z))}{ |\det( \pi( (1-t)a + tb)(z))|}, \quad t\in [0,1]  \\
   g_s(z) &= \frac{ \det( \pi(b)(sz) ) }{|\det( \pi(b)(sz) )|}, \quad s \in [0,1], 
\end{align*}
for any $z\in \T$.
Since $f_1=g_1$, $f_0$ is homotopic to $g_0$.
The winding number of $f_0$ is equal to $n$.
So it is a contradiction that $f_0$ is homotopic to the constant function $g_0$.

We shall prove ${\rm ltsr}(B_n)\le 2 $.

Let $x, y\in B_n$ and $\epsilon$ be any positive number.
We can choose $\tilde{x} = \sum_{i=0}^{n-1}\tilde{x}_i\delta^i$, $\tilde{y} = \sum_{i=0}^{n-1}\tilde{y}_i\delta^i \in B_n$ such that $\|x-\tilde{x}\|<\epsilon/2$, $\|y-\tilde{y}\|<\epsilon/2$ and
$\tilde{x}_i, \tilde{y}_i \in A(\D)$ are polynomials in z ($i=0,1,2,\ldots,n-1$).
By Lemma~\ref{lem:approximation} there exist $a, b \in B_n$ such that $\|a-\tilde{x}\|<\epsilon/2$, 
$\|b-\tilde{y}\|<\epsilon/2$, 
$a_i, b_i$ are polynomials of $z$ ($i=0,1,2,\ldots,n-1$), and
\[  \{ z\in \C : a^{(n,1)}_0 = 0 \} \cap \{ z\in \C : b^{(n,1)}_0=0 \} = \phi.  \]
So we have $\|x-a\|<\epsilon$, $\|y-b\|<\epsilon$ and
\[   a^{(n,1)} =a^{(n,1)}_0 \delta^0 \in B_n*a  \text{ and }
      b^{(n,1)} =b^{(n,1)}_0 \delta^0 \in B_n*b.  \]

Hence, there exist polynomials $p, q$ such that
$pa^{(n,1)}_0 + qb^{(n,1)}_0 =1$.
So we have
\[   \delta^0 = p\delta^0 * a^{(n,1)} + q\delta^0 * b^{(n,1)} \in B_n*a + B_n*b.   \]
This means ${\rm ltsr}(B_n)\le 2$. 

Therefore, we conclude that $\mathrm{ltsr}(B_n) = 2$.
\end{proof}

\vskip 2mm

Let $\D=\{z\in \C : |z|<1\}$ and $\varphi(z)$ be a mapping from $\D$ to $\D$.
If $\varphi(z)$ is a holomorphic function on $\D$ and a bijection from $\D$ to $\D$,
then it is known that $\varphi(z)$ has the following form:
\begin{align*}
   \varphi(z) & = e^{i\theta}\frac{z-\gamma}{1-\bar{\gamma}z} \quad \text{ for some }
 | \gamma|<1 \text{ and }  \theta\in \R  \\
                & = \frac{az+b}{\bar{b}z+\bar{a}} \quad \text{ for some } a,b\in \C \text{ with } |a|^2-|b|^2=1 .  
\end{align*}
So we can define $\alpha_\varphi \in {\rm Aut}(A(\D))$ as follows:
\[   \alpha_\varphi(f)(z) =f(\varphi(z)) \quad f\in A(\D), z\in \D,  \]
where $\alpha \in {\rm Aut}(A(\D))$ means that $\alpha: A(\D)\longrightarrow A(\D)$ is 
an algebra isomorphism and $\alpha$ and $\alpha^{-1}$ are continuous.

Conversely, for any $\varphi\in {\rm Aut}(A(\D))$ there exists a holomorphic bijection $\varphi$ on $\D$ such that
$\alpha = \alpha_{\varphi}$ (see, \cite{Hoffman 1962}, Chapter 9),

\vskip 2mm

We consider the continuous group
\[  SU(1,1) = \left\{ \begin{pmatrix} a & b \\ \bar{b} & \bar{a} \end{pmatrix} : a,b\in \C, |a|^2-|b|^2=1 \right\}  \]
and its compact subgroup
\[  U(1) = \left\{ \begin{pmatrix} a & 0 \\ 0 & \bar{a} \end{pmatrix} : a\in \C, |a|=1 \right\} . \]
For $g = \begin{pmatrix} a & b \\ \bar{b} & \bar{a} \end{pmatrix} \in SU(1,1)$, we define $\varphi_g$ and $\alpha_g$ as follows:
\[  \varphi_g(z) = \frac{az+b}{\bar{b}z+\bar{a}}  \text{ and } \alpha_g = \alpha_{\varphi_g} \in {\rm Aut}(A(\D)) . \]
We also define a continuous representation $\pi$ of $SU(1,1)$ on the $2$ dimensional real Hilbert space $\R^2$ as follows:
\[   \pi(g) = \begin{pmatrix} 1 & -i \\ 1 & i \end{pmatrix}^{-1} g  \begin{pmatrix} 1 & -i \\ 1 & i \end{pmatrix}
             = \frac{1}{2} \begin{pmatrix} 1 & 1 \\ i & -i \end{pmatrix} g  \begin{pmatrix} 1 & -i \\ 1 & i \end{pmatrix} \in M_2(\R)  \]
for $g\in SU(1,1)$.
We can get the following result by direct calculation:

\begin{lem}\label{lem:su-sl}
The representation $\pi$ is bijective from $SU(1,1)$ to \\
$SL(2,\R) = \left\{ \begin{pmatrix} s & t \\ u & v \end{pmatrix} : s,t,u,v \in \R, sv-tu=1\right\}$ and
also bijective from $U(1)$ to $SO(2) = \left\{ \begin{pmatrix} \cos \theta & -\sin \theta \\ \sin \theta & \cos \theta \end{pmatrix} : \theta\in \R \right\}$.
\end{lem}

\begin{lem}\label{lem:compact}
Let $K$ be a compact subgroup of $SU(1,1)$. 
Then there exists an element $h\in SU(1,1)$ such that $h^{-1}Kh\subset U(1)$.

In particular, any finite subgroup of $SU(1,1)$ is cyclic. 
\end{lem}
\begin{proof}
Let $\langle \cdot | \cdot \rangle$ be an usual inner product on $\R^2$.
Since the restriction $\pi|_K$ of $\pi$  is a bounded representation of $K$ on $\R^2$,
we can define new inner product on $\R^2$ as follows:
\[   [\xi|\eta] = \int_{g\in K} \langle \pi(g)\xi | \pi(g)\eta \rangle d\mu(g),   \]
where $\mu$ is the Haar measure on $K$ and $\xi, \eta \in \R^2$.
So there exists a positive $T\in GL(2,\R)$ such that $[\xi|\eta]=\langle T\xi | \eta \rangle$ for any $\xi,\eta \in \R^2$.
By definition we have $[\pi(g)\xi|\pi(g)\eta] =[\xi|\eta]$ for any $g\in K$.
This implies $T^{1/2}\pi(g)T^{-1/2} \in SO(2)$ for $g\in K$.
Put $S=T^{-1/2}/\sqrt{\det T^{-1/2}} \in SL(2,\R)$.
Since $S^{-1}\pi(g)S \in SO(2)$, there exists $h\in SU(1,1)$ with $\pi(h)=S$ and $h^{-1}Kh \subset U(1)$ by Lemma~\ref{lem:su-sl}.

The rest part is clear.
\end{proof}

\begin{thm}\label{thm:finite group}
Let $K$ be a finite group and $\alpha$ an action of $K$ on ${\rm Aut}A(\D)$. 
Then we have
\[  {\rm ltsr}(\ell^1(K, A(\D), \alpha)) =2.  \]
\end{thm}
\begin{proof}
Since $K$ is a subgroup of $PSU(1,1) = SU(1,1)/\{\pm\}$, there exists a subgroup $H$ of $SU(1,1)$ such that $H/\{\pm\} =K$.
So we may assume $K$ is a cyclic subgroup $\langle k \rangle \cong \Z/|K|\Z (= \Z_{|K|})$ of $SU(1,1)$ and $h^{-1} k h = j \in U(1)$ by Lemma~\ref{lem:compact}.

We define automorphisms of $A(\D)$ as follows:
\[   \alpha =\alpha_k, \; \beta=\alpha_j, \; \gamma = \alpha_h .  \]
Then we have $\alpha^{|K|}=\beta^{|K|}={\rm id}$ and $\alpha \gamma = \gamma \beta$.
We remark that there exists $\omega \in \C$ such that $\omega^{|K|}=1$ and
\[  \beta(f)(z) =  \sum_{l=1}^\infty c_l\omega^lz^l  \text{ for }  f(z)=\sum_{l=0}^\infty c_lz^l \in A(\D). \]
Then $\ell^1(\Z_{|K|}, A(\D), \alpha)$ and $\ell^1(\Z_{|K|}, A(\D), \beta)$ is isomorphic by the correspondence
\[  \ell^1(\Z_{|K|}, A(\D), \alpha) \ni \sum_{g\in \Z_{|K|}} f_g \delta^g \mapsto \sum_{g\in \Z_{|K|}} \gamma^{-1}(f_g) \tilde{\delta^g} \in \ell^1(\Z_{|K|}, A(\D), \beta). \]
Since $\ell^1(K, A(\D), \alpha) = \ell^1(\Z_{|K|}, A(\D), \alpha) \cong \ell^1(\Z_{|K|}, A(\D), \beta)$ and
${\rm ltsr}(\ell^1 (\Z_{|K|}, A(\D), \beta))=2$ by Theorem~\ref{thm:disc algebra},
we have $ {\rm ltsr}(\ell^1(K, A(\D), \alpha)) =2$.
\end{proof}

%
%
%



     
 \vskip 2mm
 
 \section{Acknowledgements}
The authors would like to thank Prof. A. J. Izzo for giving us  the information about a result for $\mathrm{Aut}(A(\D))$ in \cite{Hoffman 1962} and Prof. H. Nozawa for pointing out the algebraic structure of finite group actions on $\D$.
 The second author would like to thank Tamotsu Teruya for discussion about finite depth.
 The second author's research is supported by KAKENHI grant No. JP20K03644.




\end{document}